\documentclass[a4paper]{amsart}

\usepackage[all]{xy}
\usepackage{mathrsfs}

\newcommand{\dsum}{\oplus}
\newcommand{\Dsum}{\bigoplus}
\newcommand{\tensor}{\otimes}
\newcommand{\tk}{\tensor_k}
\newcommand{\tl}{\tensor_\Lambda}
\newcommand{\ts}{\tensor_\Sigma}
\newcommand{\opposite}[1]{#1^{\operatorname{op}}}
\newcommand{\e}[1]{#1^{\operatorname{e}}}
\newcommand{\iso}{\cong}
\newcommand{\equivalence}{\simeq}
\newcommand{\leftmod}[2]{{}_{#1}#2}
\newcommand{\rightmod}[2]{#1_{#2}}
\newcommand{\bimod}[3]{{}_{#1}#2_{#3}}
\newcommand{\HH}[1]{\operatorname{HH}^{#1}}
\newcommand{\fgtext}[1]{\textup{\textrm{\textbf{#1}}}}
\newcommand{\projhom}{\mathscr{P}}
\newcommand{\syzygy}{\Omega}
\newcommand{\into}{\hookrightarrow}
\newcommand{\fcomp}{\circ}
\newcommand{\extcomp}{\circ}
\newcommand{\defterm}[1]{\textit{#1}}
\newcommand{\Dsg}{\mathsf{D}_\mathsf{sg}}
\newcommand{\Db}{\mathsf{D}^\mathsf{b}}
\newcommand{\extring}[3][*]{\mathcal{E}^{#1}_{#2}(#3)}

\DeclareMathOperator{\pd}{pd}
\DeclareMathOperator{\id}{id}
\DeclareMathOperator{\Hom}{Hom}
\DeclareMathOperator{\stHom}{\underline{Hom}}
\DeclareMathOperator{\Ext}{Ext}
\DeclareMathOperator{\fmod}{mod}
\DeclareMathOperator{\stmod}{\underline{mod}}
\DeclareMathOperator{\CM}{CM}
\DeclareMathOperator{\sCM}{\underline{CM}}
\DeclareMathOperator{\im}{im}
\DeclareMathOperator{\perf}{perf}
\DeclareMathOperator{\rad}{rad}
\DeclareMathOperator{\add}{add}

\numberwithin{equation}{section}
\newtheorem{thm}{Theorem}[section]
\newtheorem{cor}[thm]{Corollary}
\newtheorem{lem}[thm]{Lemma}
\newtheorem{prop}[thm]{Proposition}
\theoremstyle{definition}
\newtheorem*{defn}{Definition}
\newtheorem{rem}[thm]{Remark}
\newtheorem{ex}[thm]{Example}

\begin{document}

\title{Singular equivalence and the (Fg) condition}
\author{\O{}ystein Skarts\ae{}terhagen}
\address{Institutt for matematiske fag, NTNU \\ N-7491 Trondheim \\ Norway}
\email{Oystein.Skartsaterhagen@math.ntnu.no}
\date{\today}



\begin{abstract}
We show that singular equivalences of Morita type with level between
finite-dimensional Gorenstein algebras over a field preserve the
\fgtext{(Fg)}~condition.
\end{abstract}

\maketitle

\setcounter{tocdepth}{1}
\tableofcontents

\section{Introduction}
\label{sec:introduction}

Throughout the paper, we let $k$ be a fixed field.

Support varieties for modules over a group algebra $kG$ were
introduced by J.~F.~Carlson in~\cite{carlson}, using the group
cohomology ring $\operatorname{H}^*(G,k)$.  Later, Snashall and
Solberg~\cite{ss} defined support varieties for modules over an
arbitrary finite-dimensional $k$-algebra $\Lambda$, using the
Hochschild cohomology ring $\HH*(\Lambda)$.

We say that a finite-dimensional $k$-algebra $\Lambda$ satisfies the
\fgtext{(Fg)}~condition if the Hochschild cohomology ring
$\HH*(\Lambda)$ of $\Lambda$ is Noetherian and the Yoneda algebra
$\Ext_\Lambda^*(M,M)$ is a finitely generated $\HH*(\Lambda)$-module
for every finitely generated $\Lambda$-module $M$ (for more details,
see the definition in Section~\ref{sec:fg}).  It was shown
in~\cite{ehsst} that many of the results for support varieties over a
group algebra also hold for support varieties over a selfinjective
algebra which satisfies the \fgtext{(Fg)}~condition.  We can thus think
of the \fgtext{(Fg)}~condition as a criterion for deciding whether a
given algebra has a nice theory of support varieties.

It is therefore interesting to investigate whether the
\fgtext{(Fg)}~condition holds for various algebras, and to find out
which relations between algebras preserve the \fgtext{(Fg)}~condition.
This question has been considered in~\cite{pss} for algebras whose
module categories are related by a recollement of abelian categories,
and in~\cite{kps} for derived equivalence of algebras.  In this paper,
we consider singular equivalence of algebras.

The \defterm{singularity category} $\Dsg(\Lambda)$ of a $k$-algebra
$\Lambda$, introduced by Buchweitz in~\cite{buchweitz}, is defined as
the Verdier quotient
\[
\Dsg(\Lambda) = \Db(\fmod \Lambda)/\perf(\Lambda)
\]
of the bounded derived category $\Db(\fmod \Lambda)$ by the
subcategory of perfect complexes.  This is a triangulated category.
We say that two $k$-algebras $\Lambda$ and~$\Sigma$ are
\defterm{singularly equivalent} if there exists a triangle equivalence
$f\colon \Dsg(\Lambda) \to \Dsg(\Sigma)$ between their singularity
categories, and the functor~$f$ is then called a \defterm{singular
  equivalence} between the algebras $\Lambda$ and~$\Sigma$.

The purpose of this paper is to investigate to what extent singular
equivalences preserve the \fgtext{(Fg)}~condition.  Since arbitrary
singular equivalences are hard to work with and do not necessarily
have nice properties, we restrict our attention to special classes
of singular equivalences.

A \defterm{singular equivalence of Morita type} (introduced by Chen
and Sun in~\cite{chen-sun}) between $k$-algebras $\Lambda$
and~$\Sigma$ is a singular equivalence
\[
\Dsg(\Lambda) \xrightarrow{N \tl -} \Dsg(\Sigma)
\]
which is induced by a tensor functor $N \tl -$, where $N$ is a
$\Sigma$--$\Lambda$ bimodule subject to some technical requirements.
Wang~\cite{wang} has introduced a generalized version of singular
equivalence of Morita type called \defterm{singular equivalence of
  Morita type with level}.  We recall the definitions of these two
types of singular equivalences in Section~\ref{sec:semtl}.  The
question we want to answer in this paper is: Do singular equivalences
of Morita type with level preserve the \fgtext{(Fg)}~condition?

All algebras that satisfy the \fgtext{(Fg)} condition are Gorenstein
algebras (see Theorem~\ref{thm:fg=>gorenstein}), and singular
equivalences of Morita type with level do not preserve Gorensteinness.
Moreover, even if one of the algebras involved in a singular
equivalence of Morita type with level satisfies the
\fgtext{(Fg)}~condition, the other algebra does not need to be a
Gorenstein algebra (see Example~\ref{ex:fg-not-preserved}).  This
means that the \fgtext{(Fg)}~condition is in general not preserved under
singular equivalence of Morita type with level.

However, we can consider the question of whether it is only when one
of the algebras is non-Gorenstein that such counterexamples arise.  In
other words, if we require all our algebras to be Gorenstein, is it
then true that singular equivalences of Morita type with level
preserve the \fgtext{(Fg)}~condition?  The main result of this paper,
Theorem~\ref{thm:main}, answers this question affirmatively: A
singular equivalence of Morita type with level between
finite-dimensional Gorenstein algebras over a field preserves the
\fgtext{(Fg)}~condition.  As a consequence of this, we obtain a
similar statement for \emph{stable} equivalence of Morita type
(Corollary~\ref{cor:stable-equivalence}), where we do not need the
assumption of Gorensteinness.

The content of the paper is structured as follows.

In Section~\ref{sec:semtl}, we state the definitions of singular
equivalence of Morita type and singular equivalence of Morita type
with level, and look at some easily derived consequences.

In Section~\ref{sec:gorenstein-cm}, we begin to look at what more we
can deduce from a singular equivalence of Morita type with level when
the assumption of Gorensteinness is added.  We recall the well-known
result stating that the singularity category of a Gorenstein algebra
is equivalent to the stable category of maximal Cohen--Macaulay
modules.  This implies that a singular equivalence
\[
f\colon \Dsg(\Lambda) \xrightarrow{\equivalence} \Dsg(\Sigma)
\]
between Gorenstein algebras $\Lambda$ and~$\Sigma$ gives
an equivalence
\[
g\colon \sCM(\Lambda) \xrightarrow{\equivalence} \sCM(\Sigma)
\]
between their stable categories of maximal Cohen--Macaulay modules.
We show that if the singular equivalence $f$ is of Morita type with
level, and thus induced by a tensor functor, then the equivalence $g$
is induced by the same tensor functor.

In Section~\ref{sec:rot}, we consider certain maps of the form
\[
\Ext_\Lambda^n(U, V)
\to \Ext_\Lambda^n(\syzygy_\Lambda^i(U), \syzygy_\Lambda^i(V)),
\]
which we call rotation maps.  We show that these maps are isomorphisms
if the algebra~$\Lambda$ is Gorenstein and $n > \id_\Lambda \Lambda$.
This means that in extension groups of sufficiently high degree over a
Gorenstein algebra, we can replace both modules by syzygies.  This
result is used in the following three sections.

In Section~\ref{sec:ext}, we show that if we have a singular
equivalence of Morita type with level
\[
\Dsg(\Lambda) \xrightarrow[N \tl -]{\equivalence} \Dsg(\Sigma)
\]
between two Gorenstein algebras $\Lambda$ and~$\Sigma$, then we have
isomorphisms
\begin{equation}
\label{eqn:intro-ext-iso}
\Ext_\Lambda^n(A, B)
\xrightarrow[N \tl -]{\iso}
\Ext_\Sigma^n(N \tl A, N \tl B)
\qquad
\text{(for $A$ and $B$ in $\fmod \Lambda$)}
\end{equation}
between extension groups over~$\Lambda$ and extension groups
over~$\Sigma$, in all sufficiently large degrees~$n$.  In the
terminology of~\cite{pss}, this implies that a tensor functor inducing
a singular equivalence of Morita type with level between Gorenstein
algebras is an eventually homological isomorphism.  The proof of this
result builds on the result about stable categories of Cohen--Macaulay
modules from Section~\ref{sec:gorenstein-cm}.

In Section~\ref{sec:hh}, we show that a singular equivalence of Morita
type with level between Gorenstein algebras preserves Hochschild
cohomology in almost all degrees.  That is, if two Gorenstein algebras
$\Lambda$ and~$\Sigma$ are singularly equivalent of Morita type with
level, then there are isomorphisms
\begin{equation}
\label{eqn:intro-hh-iso}
\HH{n}(\Lambda) \iso \HH{n}(\Sigma)
\end{equation}
for almost all~$n$, and these isomorphisms respect the ring structure
of the Hochschild cohomology.

In Section~\ref{sec:fg}, we show the main result of the paper: A
singular equivalence of Morita type with level between
finite-dimensional Gorenstein algebras over a field preserves the
\fgtext{(Fg)} condition.  The main ingredients in the proof of this
result are the isomorphism~\eqref{eqn:intro-ext-iso} of extension
groups from Section~\ref{sec:ext} and the
isomorphism~\eqref{eqn:intro-hh-iso} of Hochschild cohomology groups
from Section~\ref{sec:hh}.

\subsection*{Acknowledgments}

I would like to thank \O{}yvind Solberg and Chrysostomos Psaroudakis
for helpful discussions and suggestions.  I would also like to thank
Yiping Chen for informing me about the result stated in
Corollary~\ref{cor:stable-equivalence} and how it follows from the
main result of this paper.

\section{Singular equivalences of Morita type with level}
\label{sec:semtl}

In this section, we recall the definitions we need regarding singular
equivalences.  We begin with the concept of singularity categories.

\begin{defn}
Let $\Lambda$ be a $k$-algebra.  The \defterm{singularity category}
$\Dsg(\Lambda)$ of~$\Lambda$ is a triangulated category defined as the
Verdier quotient
\[
\Dsg(\Lambda) = \Db(\fmod \Lambda)/\perf(\Lambda)
\]
of the bounded derived category $\Db(\fmod \Lambda)$ by the
subcategory of perfect complexes.  We say that two algebras $\Lambda$
and~$\Sigma$ are \defterm{singularly equivalent} if their singularity
categories $\Dsg(\Lambda)$ and $\Dsg(\Sigma)$ are equivalent as
triangulated categories.  A~triangle equivalence between
$\Dsg(\Lambda)$ and $\Dsg(\Sigma)$ is called a \defterm{singular
  equivalence} between the algebras $\Lambda$ and~$\Sigma$.
\end{defn}

The singularity category of an algebra was first defined by Buchweitz
in \cite[Definition~1.2.2]{buchweitz}.  In his definition, the
singularity category is called the \emph{stabilized derived category},
and it is denoted by $\underline{\Db(\Lambda)}$.  Later,
Orlov~\cite{orlov} used the same construction in algebraic geometry to
define the \emph{triangulated category of singularities} of a scheme
$X$, denoted $\mathbf{D}_{Sg}(X)$.  We follow the recent convention of
using Orlov's terminology and notation for algebras as well.  The term
\emph{singular equivalence} was introduced by Chen~\cite{chen}.

Analogously to the special type of stable equivalences called
\emph{stable equivalences of Morita type}, Chen and Sun have defined a
special type of singular equivalences called \emph{singular
  equivalences of Morita type} in their preprint~\cite{chen-sun}.
This concept was further explored by Zhou and Zimmermann in~\cite{zz}.

\begin{defn}
Let $\Lambda$ and~$\Sigma$ be finite-dimensional $k$-algebras, and let
$M$ be a $\Lambda$--$\Sigma$ bimodule and $N$ a $\Sigma$--$\Lambda$
bimodule.  We say that $M$ and~$N$ induce a \defterm{singular
  equivalence of Morita type} between $\Lambda$ and~$\Sigma$ (and that
$\Lambda$ and~$\Sigma$ are \defterm{singularly equivalent of Morita
  type}) if the following conditions are satisfied:
\begin{enumerate}
\item $M$ is finitely generated and projective as a left
$\Lambda$-module and as a right $\Sigma$-module.
\item $N$ is finitely generated and projective as a left
$\Sigma$-module and as a right $\Lambda$-module.
\item There is a finitely generated $\e{\Lambda}$-module $X$ with
finite projective dimension such that $M \ts N \iso \Lambda \dsum X$
as $\e{\Lambda}$-modules.
\item There is a finitely generated $\e{\Sigma}$-module $Y$ with
finite projective dimension such that $N \tl M \iso \Sigma \dsum Y$ as
$\e{\Sigma}$-modules.
\end{enumerate}
\end{defn}

Notice that the definition is precisely the same as the definition of
stable equivalence of Morita type, except that the modules $X$ and~$Y$
are not necessarily projective, but only have finite projective
dimension.  Thus stable equivalences of Morita type occur as a special
case of singular equivalences of Morita type.

The following proposition describes how a singular equivalence of
Morita type is a singular equivalence, thus justifying the name.

\begin{prop}\cite[Proposition 2.3]{zz}
\label{prop:semt-equivalence}
Let $\bimod{\Lambda}{M}{\Sigma}$ and~$\bimod{\Sigma}{N}{\Lambda}$ be
bimodules which induce a singular equivalence of Morita type between
two $k$-algebras $\Lambda$ and $\Sigma$.  Then the functors
\[
N \tl - \colon \Dsg(\Lambda) \to \Dsg(\Sigma)
\qquad\text{and}\qquad
M \ts - \colon \Dsg(\Sigma) \to \Dsg(\Lambda)
\]
are equivalences of triangulated categories, and they are
quasi-inverses of each other.
\end{prop}

Inspired by the notion of singular equivalence of Morita type,
Wang~\cite{wang} has defined a more general type of singular
equivalence called \emph{singular equivalence of Morita type with
  level}.

\begin{defn}
Let $\Lambda$ and~$\Sigma$ be finite-dimensional $k$-algebras, and let
$M$ be a $\Lambda$--$\Sigma$ bimodule and $N$ a $\Sigma$--$\Lambda$
bimodule.  Let $l$ be a nonnegative integer.  We say that $M$ and~$N$
induce a \defterm{singular equivalence of Morita type with level~$l$}
between $\Lambda$ and~$\Sigma$ (and that $\Lambda$ and~$\Sigma$ are
\defterm{singularly equivalent of Morita type with level~$l$}) if the
following conditions are satisfied:
\begin{enumerate}
\item $M$ is finitely generated and projective as a left
$\Lambda$-module and as a right $\Sigma$-module.
\item $N$ is finitely generated and projective as a left
$\Sigma$-module and as a right $\Lambda$-module.
\item There is an isomorphism $M \ts N \iso
\syzygy_{\e\Lambda}^l(\Lambda)$ in the stable category $\stmod
\e\Lambda$.
\item There is an isomorphism $N \tl M \iso
\syzygy_{\e\Sigma}^l(\Sigma)$ in the stable category $\stmod
\e\Sigma$.
\end{enumerate}
\end{defn}

Just as in the case of singular equivalence of Morita type, the
conditions in the definition of singular equivalence of Morita type
with level are designed to ensure that the functors $N \tl -$ and $M
\ts -$ induce singular equivalences.

\begin{prop}\cite[Remark~2.2]{wang}
\label{prop:semtl-equivalence}
Let $\bimod{\Lambda}{M}{\Sigma}$ and~$\bimod{\Sigma}{N}{\Lambda}$ be
bimodules which induce a singular equivalence of Morita type with
level~$l$ between two $k$-algebras $\Lambda$ and~$\Sigma$.  Then the
functors
\[
N \tl - \colon \Dsg(\Lambda) \to \Dsg(\Sigma)
\qquad\text{and}\qquad
M \ts - \colon \Dsg(\Sigma) \to \Dsg(\Lambda)
\]
are equivalences of triangulated categories.  The compositions
\[
M \ts N \tl - \colon \Dsg(\Lambda) \to \Dsg(\Lambda)
\qquad\text{and}\qquad
N \tl M \ts - \colon \Dsg(\Sigma) \to \Dsg(\Sigma)
\]
are isomorphic to the shift functor $[-l]$ on the respective
categories $\Dsg(\Lambda)$ and $\Dsg(\Sigma)$.
\end{prop}

We now show that the notion of singular equivalence of Morita type
with level generalizes the notion of singular equivalence of Morita
type, in the sense that any equivalence of the latter type is also of
the former type.  This is mentioned without proof in~\cite{wang}.

\begin{prop}
\label{prop:semt=>semtl}
Let $\Lambda$ and~$\Sigma$ be finite-dimensional $k$-algebras.  If a
functor $f \colon \Dsg(\Lambda) \to \Dsg(\Sigma)$ is a singular
equivalence of Morita type, then it is also a singular equivalence of
Morita type with level.
\end{prop}
\begin{proof}
Let $M$, $N$, $X$ and~$Y$ be bimodules satisfying the requirements of
a singular equivalence of Morita type, such that $f = (N \tl -)$.  Let
$l = \max \{ \pd_{\e\Lambda} X, \pd_{\e\Sigma} Y \}$.  Let $M'$ be
an $l$-th syzygy of $M$ as $\Lambda$--$\Sigma$-bimodule, and let
\begin{equation}
\label{eq:M'}
0 \to M' \to P_{l-1} \to \cdots \to P_0 \to M \to 0
\end{equation}
be the beginning of a projective resolution of~$M$.  We show that the
bimodules $M'$ and~$N$ induce a singular equivalence of Morita type
with level~$l$.

If we consider the bimodules in sequence~\eqref{eq:M'} as one-sided
modules (left $\Lambda$-modules or right $\Sigma$-modules), then $M$
and the modules $P_0, \ldots, P_{l-1}$ are projective, and thus $M'$
must be projective as well.  Thus condition~(1) in the definition is
satisfied.  Condition~(2) is trivially satisfied, since it is the same
as condition~(2) in the definition of singular equivalence of Morita
type.

Tensoring sequence~\eqref{eq:M'} with~$N$ gives the sequence
\[
0 \to M' \ts N \to P_{l-1} \ts N \to \cdots \to P_0 \ts N \to M \ts N \to 0.
\]
This sequence is exact since $N$ is projective as left
$\Sigma$-module, and the modules $P_i \ts N$ are projective
$\e\Lambda$-modules since $N$ is projective as right $\Lambda$-module.
The $\e\Lambda$-module $M' \ts N$ is therefore an $l$-th syzygy of $M \ts N$.
Since $M \ts N$ is isomorphic to $\Lambda \dsum X$ and the projective
dimension of $X$ is at most $l$, this means that $M' \ts N$ is an
$l$-th syzygy of $\Lambda$ as $\e\Lambda$-module.  Similarly, we can
show that $N \tl M'$ is an $l$-th syzygy of $\Sigma$ as
$\e\Sigma$-module.  This means that conditions (3)~and~(4) in the
definition are satisfied.
\end{proof}

In the rest of the paper we work with singular equivalences of Morita
type with level.  By the above proposition, all results where we
assume such an equivalence are also applicable to singular
equivalences of Morita type.

As seen above, if $\bimod{\Lambda}{M}{\Sigma}$
and~$\bimod{\Sigma}{N}{\Lambda}$ are bimodules which induce a singular
equivalence of Morita type with level, then the functors $N
\tensor_\Lambda -$ and $M \tensor_\Sigma -$ are equivalences between
the singularity categories of $\Lambda$ and~$\Sigma$.  We end this
section by examining some properties of these tensor functors when
viewed as functors between the module categories $\fmod \Lambda$ and
$\fmod \Sigma$.

\begin{lem}
\label{lem:semtl-functor-properties}
Let $\bimod{\Lambda}{M}{\Sigma}$ and~$\bimod{\Sigma}{N}{\Lambda}$ be
bimodules which induce a singular equivalence of Morita type with
level between two $k$-algebras $\Lambda$ and~$\Sigma$.  Then the
functors
\[
N \tl - \colon \fmod \Lambda \to \fmod \Sigma
\qquad\text{and}\qquad
M \ts - \colon \fmod \Sigma \to \fmod \Lambda
\]
are exact and take projective modules to projective modules.  In
particular, this means that they take projective resolutions to
projective resolutions.
\end{lem}
\begin{proof}
Consider the functor $N \tl -$.  This functor is exact
since $N$ is projective as right $\Lambda$-module, and it takes
projective modules to projective modules since $N$ is projective as
left $\Sigma$-module.
\end{proof}

Let $\Lambda$, $\Sigma$, $M$ and~$N$ be as in the above lemma.  Since
the functor
\[
N \tl - \colon \fmod \Lambda \to \fmod \Sigma
\]
is exact, it induces homomorphisms of extension groups.  By abuse of
notation, we denote these maps by $N \tl -$ as well.  More precisely,
for $\Lambda$-modules $U$ and~$V$ and an integer $n \ge 0$, we define
a map
\begin{equation}
\label{eq:ext-map}
N \tl - \colon
\Ext_\Lambda^n(U, V)
\to \Ext_\Sigma^n(N \tl U, N \tl V).
\end{equation}
For $n = 0$, the map $N \tl -$ simply sends a homomorphism $f \colon U
\to V$ to the homomorphism $N \tl f \colon N \tl U \to N \tl V$.  For
$n > 0$, the map $N \tl -$ sends the element represented by the
extension
\[
0 \to V \to E_n \to \cdots \to E_1 \to U \to 0
\]
to the element represented by the extension
\[
0 \to N \tl V \to N \tl E_n \to \cdots \to N \tl E_1 \to N \tl U \to 0
\]
obtained by applying the functor $N \tl -$ to all objects and maps.

The maps~\eqref{eq:ext-map} play an important role later in the paper.
In Section~\ref{sec:ext}, we show that if $\Lambda$ and~$\Sigma$ are
Gorenstein algebras, then these maps are isomorphisms for almost
all~$n$.  This fact is used in the proof of the main theorem
(Theorem~\ref{thm:main}).

\section{Gorenstein algebras and maximal Cohen--Macaulay modules}
\label{sec:gorenstein-cm}

So far, we have considered the situation of two $k$-algebras $\Lambda$
and~$\Sigma$, together with bimodules $\bimod{\Lambda}{M}{\Sigma}$
and~$\bimod{\Sigma}{N}{\Lambda}$ inducing a singular equivalence of
Morita type with level between $\Lambda$ and~$\Sigma$.  From now on,
we restrict our attention to the special case where both $\Lambda$
and~$\Sigma$ are Gorenstein algebras.  In this section, we prove our
first result under this assumption, namely
Proposition~\ref{prop:semtl->sCM}, which states that the tensor
functors $N \tl -$ and $M \ts -$ induce triangle equivalences between
the stable categories of maximal Cohen--Macaulay modules over
$\Lambda$ and~$\Sigma$.

We begin by recalling the definition of Gorenstein algebras.

\begin{defn}
A $k$-algebra $\Lambda$ is a \defterm{Gorenstein algebra} if the
injective dimension of $\Lambda$ as a left $\Lambda$-module is finite
and the injective dimension of $\Lambda$ as a right $\Lambda$-module
is finite:
\[
\id_\Lambda ({}_\Lambda \Lambda) < \infty
\qquad\text{and}\qquad
\id_{\opposite{\Lambda}} (\Lambda_\Lambda) < \infty.
\]
\end{defn}


If $\Lambda$ is a Gorenstein algebra, then $\id_\Lambda
(\leftmod{\Lambda}{\Lambda})$ and $\id_{\opposite{\Lambda}}
(\rightmod{\Lambda}{\Lambda})$ are the same number, and this number is
called the \defterm{Gorenstein dimension} of $\Lambda$.  In later
sections, we need the following result about Gorenstein algebras.

\begin{lem}\cite[Lemma~2.1]{bj}
\label{lem:gorenstein-envalg}
If $\Lambda$ is a Gorenstein $k$-algebra with Gorenstein dimension
$d$, then its enveloping algebra $\e\Lambda$ is a Gorenstein algebra
with Gorenstein dimension at most $2d$.
\end{lem}

We continue by recalling the definition of maximal Cohen--Macaulay
modules.

\begin{defn}
Let $\Lambda$ be a $k$-algebra.  A finitely generated $\Lambda$-module
$C$ is a \defterm{maximal Cohen--Macaulay module} if
$\Ext_\Lambda^n(C, \Lambda) = 0$ for every positive integer $n$.  We
denote the subcategory of $\fmod \Lambda$ consisting of all maximal
Cohen--Macaulay modules by $\CM(\Lambda)$, and the corresponding
stable category modulo projectives by $\sCM(\Lambda)$.
\end{defn}

In the following lemma, we recall some characterizations of maximal
Cohen--Macaulay modules over Gorenstein algebras.

\begin{lem}
\label{lem:cm<->syzygy}
Let $\Lambda$ be a finite-dimensional Gorenstein $k$-algebra and $C$ a
finitely generated $\Lambda$-module.  The following are equivalent.
\begin{enumerate}
\item $C$ is a maximal Cohen--Macaulay module.
\item $C$ has a projective coresolution.  That is, there exists an
exact sequence
\[
0 \to C \to P_{-1} \to P_{-2} \to \cdots
\]
where every $P_i$ is a projective $\Lambda$-module.
\item For every $n > 0$, there is a $\Lambda$-module $A$ such that $C$
is an $n$-th syzygy of $A$.
\item For some $n \ge \id_\Lambda \Lambda$, there is a
$\Lambda$-module $A$ such that $C$ is an $n$-th syzygy of $A$.
\end{enumerate}
\end{lem}
\begin{proof}
We only need to show that statement~(1) implies statement~(2); the
implications $\text{(2)} \implies \text{(3)} \implies \text{(4)}$ are
obvious, and the implication $\text{(4)} \implies \text{(1)}$ follows
directly from the definitions.

We use Theorem~5.4~(b) from~\cite{contravariantly}.  We first describe
the notation used in~\cite{contravariantly} for certain subcategories
of a module category.

For a $\Lambda$-module $T$ with the property that $\Ext_\Lambda^i(T,T)
= 0$ for every $i>0$, we define the subcategories ${}^{\bot}T$
and~$\mathscr{X}_T$ of $\fmod \Lambda$.  The category ${}^{\bot}T$ is
the subcategory of $\fmod \Lambda$ consisting of all modules $A$ such
that $\Ext_\Lambda^i(A,T) = 0$ for every $i>0$.  The category
$\mathscr{X}_T$ is the subcategory of ${}^{\bot}T$ consisting of all
modules $A$ such that there is an exact sequence
\[
0 \to A
  \to T_0
  \xrightarrow{f_0} T_1
  \xrightarrow{f_1} T_2
  \xrightarrow{f_2} \cdots
\]
where $T_i$ is in $\add T$ and $\im f_i$ is in ${}^{\bot}T$ for every
$i \ge 0$.

Theorem~5.4~(b) in~\cite{contravariantly} says that if $T$ is a
cotilting module, then the categories ${}^{\bot}T$ and~$\mathscr{X}_T$
are equal.

Now consider the case $T = \Lambda$.  Since $\Lambda$ is a Gorenstein
algebra, it is a cotilting module, and then by the above we have
${}^{\bot}\Lambda = \mathscr{X}_\Lambda$.  Furthermore,
${}^{\bot}\Lambda$ is the category $\CM(\Lambda)$ of maximal Cohen--Macaulay modules.
Therefore, every maximal Cohen--Macaulay module is in the category
$\mathscr{X}_\Lambda$, and thus it has a resolution of the form
\[
0 \to C \to P_{-1} \to P_{-2} \to \cdots
\]
where every $P_i$ is a projective $\Lambda$-module.
\end{proof}

We now recall the theorem by Buchweitz which provides the connection
we need between singularity categories and stable categories of
maximal Cohen--Macaulay modules.

\begin{thm}\cite[Theorem~4.4.1]{buchweitz}
\label{thm:sCM-equiv-Dsg}
Let $\Lambda$ be a finite-dimensional Gorenstein algebra.  Then there
is an equivalence of triangulated categories
\[
\sCM(\Lambda) \xrightarrow{\equivalence} \Dsg(\Lambda)
\]
given by sending every object in $\sCM(\Lambda)$ to a stalk complex
concentrated in degree~$0$.
\end{thm}

A direct consequence of Theorem~\ref{thm:sCM-equiv-Dsg} is that if two
finite-dimensional Gorenstein algebras $\Lambda$ and~$\Sigma$ are
singularly equivalent, then the categories $\sCM(\Lambda)$ and
$\sCM(\Sigma)$ are triangle equivalent.  If the algebras are not only
singularly equivalent, but singularly equivalent of Morita type (with
level), then there are tensor functors $N \tl -$ and $M \ts -$ that
induce equivalences between the singularity categories $\Dsg(\Lambda)$
and $\Dsg(\Sigma)$.  What we aim to prove now is that these tensor
functors also induce equivalences between the stable categories
$\sCM(\Lambda)$ and $\sCM(\Sigma)$ of maximal Cohen--Macaulay modules.
We first show that these functors preserve the property of being a
maximal Cohen--Macaulay module.

\begin{lem}
\label{lem:tensor-functor-cm}
Let $\bimod{\Lambda}{M}{\Sigma}$ and~$\bimod{\Sigma}{N}{\Lambda}$ be
bimodules which induce a singular equivalence of Morita type with
level between two finite-dimensional Gorenstein $k$-algebras $\Lambda$
and~$\Sigma$.  Then the functors
\[
N \tl - \colon \fmod \Lambda \to \fmod \Sigma
\qquad\text{and}\qquad
M \ts - \colon \fmod \Sigma \to \fmod \Lambda
\]
send maximal Cohen--Macaulay modules to maximal Cohen--Macaulay
modules.
\end{lem}
\begin{proof}
Let $n = \max \{ \id_\Lambda \Lambda, \id_\Sigma \Sigma \}$ (this is
finite since the algebras $\Lambda$ and~$\Sigma$ are Gorenstein).  Let
$C$ be a maximal Cohen--Macaulay module over $\Lambda$.  Then by
Lemma~\ref{lem:cm<->syzygy}, there is a $\Lambda$-module $A$ such that
$C$ is an $n$-th syzygy of $A$.  By
Lemma~\ref{lem:semtl-functor-properties}, the $\Sigma$-module $N \tl
C$ is an $n$-th syzygy of $N \tl A$, and therefore by
Lemma~\ref{lem:cm<->syzygy} it is a maximal Cohen--Macaulay module.
\end{proof}

Finally, we are ready to prove the main the result of this section.

\begin{prop}
\label{prop:semtl->sCM}
Let $\bimod{\Lambda}{M}{\Sigma}$ and~$\bimod{\Sigma}{N}{\Lambda}$ be
bimodules which induce a singular equivalence of Morita type with
level between two finite-dimensional Gorenstein $k$-algebras $\Lambda$
and~$\Sigma$.  Then the functors
\[
N \tl - \colon \sCM(\Lambda) \to \sCM(\Sigma)
\qquad\text{and}\qquad
M \ts - \colon \sCM(\Sigma) \to \sCM(\Lambda)
\]
are equivalences of triangulated categories.
\end{prop}
\begin{proof}
We first check that $N \tl -$ actually gives a functor from
$\sCM(\Lambda)$ to $\sCM(\Sigma)$.  We know from
Lemma~\ref{lem:tensor-functor-cm} that it gives a functor from
$\CM(\Lambda)$ to $\CM(\Sigma)$.  By
Lemma~\ref{lem:semtl-functor-properties}, we see that if $f$ is a map
of $\Lambda$-modules that factors through a projective module, then
the map $N \tl f$ also factors through a projective module.  Thus $N
\tl -$ gives a well-defined functor from $\sCM(\Lambda)$ to
$\sCM(\Sigma)$.

Consider the diagram
\[
\xymatrix@C=5em{
\sCM(\Lambda)
  \ar[r]^{N \tl -}
  \ar[d]^{\equivalence}
&
\sCM(\Sigma)
  \ar[d]^{\equivalence}
\\
\Dsg(\Lambda)
  \ar[r]^{N \tl -}_{\equivalence}
&
\Dsg(\Sigma)
}
\]
of categories and functors, where the vertical functors are the
equivalences from Theorem~\ref{thm:sCM-equiv-Dsg}, and the functor $N
\tl -$ in the bottom row is an equivalence by
Proposition~\ref{prop:semtl-equivalence}.  The diagram commutes, and
therefore the functor $N \tl -$ in the top row is also an
equivalence.
\end{proof}

\section{Rotations of extensions}
\label{sec:rot}

If $U$ and~$V$ are modules over an algebra $\Lambda$, then dimension
shift gives isomorphisms $\Ext_\Lambda^n(U, V) \iso
\Ext_\Lambda^{n-i}(\syzygy_\Lambda^i(U), V)$ for integers $n$ and~$i$
with $n > i > 0$.  If the algebra $\Lambda$ is Gorenstein, then all
projective $\Lambda$-modules have finite injective dimension.  This
means that for sufficiently large~$n$ (more precisely, $n >
\id_\Lambda \Lambda$), we can use projective resolutions to do
dimension shifting in the second argument of Ext as well.  That is, we
have isomorphisms $\Ext_\Lambda^n(U, V) \iso \Ext_\Lambda^{n+i}(U,
\syzygy_\Lambda^i(V))$.  By dimension shifting in both arguments, we
then get isomorphisms
\[
\Ext_\Lambda^n(U, V)
\iso \Ext_\Lambda^n(\syzygy_\Lambda^i(U), \syzygy_\Lambda^i(V)),
\]
where we stay in the same degree~$n$, but replace both arguments to
Ext by their $i$-th syzygies.  In this section, we describe such
isomorphisms, which we call rotation maps, and which are going to be
used several times in later sections.

For defining the rotation maps, we do not need to assume that we are
working over a Gorenstein algebra.  This however means that the maps
are not necessarily isomorphisms.  We first define the maps in a
general setting, and then in Lemma~\ref{lem:rotation} describe the
conditions we need for ensuring that they are isomorphisms.

\begin{defn}
Let $\Lambda$ be a finite-dimensional $k$-algebra, and let $U$ and~$V$
be finitely generated $\Lambda$-modules.  Choose projective
resolutions $\pi \colon \cdots \to P_1 \to P_0 \to U \to 0$ and $\tau
\colon \cdots \to Q_1 \to Q_0 \to V \to 0$ of the modules $U$ and~$V$.
Let $i$ and~$n$ be integers with $i < n$, and let
\begin{align*}
\pi_i &\colon
0 \to \syzygy_\Lambda^i(U) \to P_{i-1} \to \cdots \to P_0 \to U \to 0, \\
\tau_i &\colon
0 \to \syzygy_\Lambda^i(V) \to Q_{i-1} \to \cdots \to Q_0 \to V \to 0
\end{align*}
be truncations of the chosen projective resolutions.  We define the
\defterm{$i$-th rotation} of the extension group $\Ext_\Lambda^n(U,
V)$ with respect to the resolutions $\pi$ and~$\tau$ to be the map
\[
\xymatrix{
{\rho_i \colon \Ext_\Lambda^n(U, V)}
\ar[rr] \ar[dr]_{(\tau_i)_*} &&
{\Ext_\Lambda^n(\syzygy_\Lambda^i(U), \syzygy_\Lambda^i(V))} \\
&
{\Ext_\Lambda^{n+i}(U, \syzygy_\Lambda^i(V))}
\ar[ur]_{(\pi_i^*)^{-1}}
}
\]
given by $\rho_i = (\pi_i^*)^{-1} (\tau_i)_*$.
\end{defn}

Consider the situation in the above definition.  If the algebra
$\Lambda$ is Gorenstein and $n > \id_\Lambda \Lambda$, then for each
of the projective modules~$Q_j$, we have $\id_\Lambda Q_j \le
\id_\Lambda \Lambda < n$, and thus the map $(\tau_i)_*$ is an
isomorphism.  This gives the following result.

\begin{lem}
\label{lem:rotation}
Let $\Lambda$ be a finite-dimensional Gorenstein $k$-algebra, and let
$U$ and~$V$ be finitely generated $\Lambda$-modules.  For every $n >
\id_\Lambda \Lambda$ and every $i < n$, the $i$-th rotation
\[
\rho_i \colon
\Ext_\Lambda^n(U, V) \to
\Ext_\Lambda^n(\syzygy_\Lambda^i(U), \syzygy_\Lambda^i(V))
\]
\textup{(}with respect to any projective resolutions of $U$
and~$V$\textup{)} is an isomorphism.
\end{lem}

If we look at a rotation map of an extension group
$\Ext_\Lambda^n(U,U)$ with the same module in both arguments, then the
action of the map can be viewed as a concrete ``rotation'' of the
extensions, as we will now see.  Let $\pi \colon \cdots \to P_1 \to
P_0 \to U \to 0$ be a projective resolution of~$U$, and consider the
$i$-th rotation map
\[
\rho_i \colon \Ext_\Lambda^n(U, U)
\to \Ext_\Lambda^n(\syzygy_\Lambda^i(U), \syzygy_\Lambda^i(U))
\]
with respect to the resolution $\pi$.  Every element of
$\Ext_\Lambda^n(U,U)$ can be represented by an exact sequence of the
form
\[
\xymatrix@C=2.5ex@R=1ex{
0 \ar[r] &
U \ar[r] &
E \ar[r] &
P_{n-2} \ar[r] &
\cdots \ar[r] &
P_i \ar[rr]\ar[rd] &&
P_{i-1} \ar[r] &
\cdots \ar[r] &
P_0 \ar[r] &
U \ar[r] &
0 \\
&&&&&& \syzygy_\Lambda^i(U) \ar[ur]
}
\]
Applying the map~$\rho_i$ to the element represented by this sequence
produces the element represented by the following sequence:
\[
\xymatrix@C=2.5ex@R=1ex{
0 \ar[r] &
\syzygy_\Lambda^i(U) \ar[r] &
P_{i-1} \ar[r] &
\cdots \ar[r] &
P_0 \ar[rr]\ar[rd] &&
E \ar[r] &
P_{n-2} \ar[r] &
\cdots \ar[r] &
P_i \ar[r] &
\syzygy_\Lambda^i(U) \ar[r] &
0 \\
&&&&& U \ar[ur]
}
\]
We have thus rotated the sequence by removing an $i$-fold sequence
from the right side and moving it to the left side.

\section{Isomorphisms between extension groups} 
\label{sec:ext}

In this section, we show that if $\Lambda$ and~$\Sigma$ are Gorenstein
algebras which are singularly equivalent of Morita type with level,
then we have isomorphisms between extension groups over $\Lambda$ and
extension groups over $\Sigma$ in sufficiently high degrees.  More
precisely, if $\bimod{\Lambda}{M}{\Sigma}$
and~$\bimod{\Sigma}{N}{\Lambda}$ are bimodules which induce a singular
equivalence of Morita type with level between the algebras $\Lambda$
and~$\Sigma$, then the functor $N \tl -$ induces an isomorphism
\begin{equation}
\label{eq:ext-iso}
\Ext_\Lambda^n(A,B)
\iso
\Ext_\Sigma^n(N \tl A,
              N \tl B)
\end{equation}
for every $n \ge \max\{\id_\Lambda \Lambda, \id_\Sigma \Sigma\}$ and
for any $\Lambda$-modules $A$ and~$B$.  This is stated as
Proposition~\ref{prop:ext-iso}.

To prove this result, we use maximal Cohen--Macaulay modules and the
results from Section~\ref{sec:gorenstein-cm}, as well as the rotation
maps from Section~\ref{sec:rot}.  By
Proposition~\ref{prop:semtl->sCM}, we know that in the setting
described above, we have isomorphisms
\begin{equation}
\label{eq:stHom-iso}
\stHom_\Lambda(C,C') \iso
\stHom_\Sigma(N \tl C,
              N \tl C')
\end{equation}
between stable $\Hom$ groups over $\Lambda$ and~$\Sigma$ for maximal
Cohen--Macaulay $\Lambda$-modules $C$ and~$C'$.
Lemma~\ref{lem:stHom=ext} below relates stable $\Hom$ groups to
extension groups.  Using this and isomorphism~\eqref{eq:stHom-iso}, we
show (Proposition~\ref{prop:ext-iso-cm}) that there are isomorphisms
\[
\Ext_\Lambda^n(C, C') \iso
\Ext_\Sigma^n(N \tl C,
              N \tl C')
\]
for all maximal Cohen--Macaulay modules $C$ and $C'$ and every
positive integer~$n$.  Finally, to arrive at
isomorphism~\eqref{eq:ext-iso} for any $\Lambda$-modules $A$ and~$B$
in Proposition~\ref{prop:ext-iso}, we use
Proposition~\ref{prop:ext-iso-cm} together with two facts about
Gorenstein algebras from earlier sections: all syzygies of
sufficiently high degree are maximal Cohen--Macaulay modules, and by
using a rotation map, we can replace the modules $A$ and~$B$ by their
syzygies.

We begin this section by showing, in the following two lemmas, how
extension groups between maximal Cohen--Macaulay modules can be
described as stable Hom groups.  If $C$ and~$C'$ are maximal
Cohen--Macaulay modules over an algebra $\Lambda$, then we get
(Lemma~\ref{lem:stHom=ext}) an isomorphism
\[
\Ext_\Lambda^n(C, C') \iso \stHom_\Lambda(K_n, C')
\]
for every positive integer~$n$, with $K_n$ an $n$-th syzygy of~$C$.

In fact, it turns out that the conditions on $C$ and~$C'$ can be
relaxed somewhat.  Recall that $C$ being a maximal Cohen--Macaulay
module means that $\Ext_\Lambda^i(C,\Lambda) = 0$ for every positive
integer~$i$.  To get the above isomorphism in degree~$n$, it is
sufficient to assume that $\Ext_\Lambda^n(C,\Lambda) = 0$, and we do
not need to put any assumptions on the module~$C'$.  We use this
weaker assumption in the lemmas.

The following notation is used in the two lemmas.  Given two modules
$A$ and $B$ over an algebra~$\Lambda$, we write $\projhom_\Lambda(A,B)
\subseteq \Hom_\Lambda(A,B)$ for the subspace of $\Hom_\Lambda(A,B)$
consisting of morphisms that factor through a projective module; then
the stable Hom group is $\stHom_\Lambda(A, B) = \Hom_\Lambda(A, B) /
\projhom_\Lambda(A, B)$.

In the first lemma, we consider the special case $n = 1$.

\begin{lem}
\label{lem:stHom=ext1}
Let $\Lambda$ be a finite-dimensional $k$-algebra, and let $A$ and~$C$
be finitely generated $\Lambda$-modules such that $\Ext_\Lambda^1(C,
\Lambda) = 0$.  Let
\[
\eta \colon 0 \to K \xrightarrow{\alpha} P \xrightarrow{\beta} C \to 0.
\]
be a short exact sequence of $\Lambda$-modules with $P$ projective.
Then the sequence
\[
0 \to \projhom_\Lambda(K, A)
  \into \Hom_\Lambda(K, A)
  \xrightarrow{\eta^*} \Ext_\Lambda^1(C, A)
  \to 0
\]
of $k$-vector spaces is exact.
\end{lem}
\begin{proof}
By applying the functor $\Hom_\Lambda(-, A)$ to the sequence $\eta$,
we get the exact sequence
\[
0 \to \Hom_\Lambda(C, A)
  \xrightarrow{\beta^*} \Hom_\Lambda(P, A)
  \xrightarrow{\alpha^*} \Hom_\Lambda(K, A)
  \xrightarrow{\eta^*} \Ext_\Lambda^1(C, A)
  \to 0.
\]
From this we obtain the short exact sequence
\[
0 \to \im \alpha^*
  \into \Hom_\Lambda(K, A)
  \xrightarrow{\eta^*} \Ext_\Lambda^1(C, A)
  \to 0.
\]

Now we only need to show that $\im \alpha^* = \projhom_\Lambda(K, A)$.
If a homomorphism $f \colon K \to A$ lies in $\im \alpha^*$, then it
factors through the map $\alpha \colon K \to P$, and since the module
$P$ is projective, this means that $f$ lies in $\projhom_\Lambda(K,
A)$.  We thus have $\im \alpha^* \subseteq \projhom_\Lambda(K, A)$.

For the opposite inclusion, let $Q$ be a projective $\Lambda$-module.
Since we have assumed that $\Ext_\Lambda^1(C, \Lambda) = 0$, we also
have $\Ext_\Lambda^1(C, Q) = 0$.  Then from the long exact sequence
obtained by applying the functor $\Hom_\Lambda(-,Q)$ to the short
exact sequence $\eta$, we see that every homomorphism $g\colon K \to
Q$ factors through the homomorphism $\alpha \colon K \to P$.  Thus
every homomorphism which starts in $K$ and factors through some
projective module, also factors through $\alpha$, and we get
$\projhom_\Lambda(K,A) \subseteq \im \alpha^*$.
\end{proof}

Now we continue to extension groups in arbitrary degree by using the
above lemma and dimension shifting.

\begin{lem}
\label{lem:stHom=ext}
Let $\Lambda$ be a finite-dimensional $k$-algebra, let $A$ and~$C$ be
finitely generated $\Lambda$-modules, and let $n$ be a positive
integer.  Assume that $\Ext_\Lambda^n(C, \Lambda) = 0$.  Let
\[
\pi_n \colon 0 \to K_n \to P_{n-1}
               \to P_{n-2} \to \cdots \to P_1 \to P_0 \to C \to 0
\]
be the beginning of a projective resolution of~$C$ with $K_n$ as the
$n$-th syzygy.  Then the sequence
\[
0 \to \projhom_\Lambda(K_n, A)
  \into \Hom_\Lambda(K_n, A)
  \xrightarrow{\pi_n^*} \Ext_\Lambda^n(C, A)
  \to 0
\]
of $k$-vector spaces is exact, and thus the map $\pi_n^*$ induces an
isomorphism
\[
\overline{\pi_n^*} \colon
\stHom_\Lambda(K_n, A) \xrightarrow{\iso} \Ext_\Lambda^n(C, A).
\]
\end{lem}
\begin{proof}
Decompose the sequence $\pi_n$ into two exact sequences
\begin{align*}
\eta \colon & 0 \to K_n \to P_{n-1} \to K_{n-1} \to 0 \\
\text{and}\qquad
\pi_{n-1} \colon & 0 \to K_{n-1} \to P_{n-2} \to P_{n-3} \to \cdots
                     \to P_1 \to P_0 \to C \to 0,
\end{align*}
such that $\pi_n = \eta \extcomp \pi_{n-1}$.  By dimension shifting, we
have an isomorphism
\[
\pi_{n-1}^* \colon \Ext_\Lambda^1(K_{n-1}, A) \xrightarrow{\iso} \Ext_\Lambda^n(C, A).
\]
We observe that $\pi_n^* = \pi_{n-1}^* \fcomp \eta^*$, so the following
diagram is commutative.
\[
\xymatrix{
0 \ar[r] &
\projhom_\Lambda(K_n, A) \ar@{^{(}->}[r] \ar@{=}[d] &
\Hom_\Lambda(K_n, A) \ar[r]^{\eta^*} \ar@{=}[d] &
\Ext_\Lambda^1(K_{n-1}, A) \ar[r] \ar[d]_{\pi_{n-1}^*}^{\iso} &
0 \\
0 \ar[r] &
\projhom_\Lambda(K_n, A) \ar@{^{(}->}[r] &
\Hom_\Lambda(K_n, A) \ar[r]^{\pi_n^*} &
\Ext_\Lambda^n(C, A) \ar[r] &
0
}
\]
By Lemma~\ref{lem:stHom=ext1}, the top row of this diagram is exact.
Since all the vertical maps are isomorphisms, the bottom row is also
exact.
\end{proof}

We now show that we get the isomorphisms we want between extension
groups in the special case where the involved modules are maximal
Cohen--Macaulay modules.  In this case, we get isomorphisms between
extension groups in all positive degrees, while in the general case
which is considered afterwards (Proposition~\ref{prop:ext-iso}), we
only get isomorphisms in almost all degrees.

\begin{prop}
\label{prop:ext-iso-cm}
Let $\Lambda$ and~$\Sigma$ be finite-dimensional Gorenstein algebras
which are singularly equivalent of Morita type with level, and let
$\bimod{\Lambda}{M}{\Sigma}$ and~$\bimod{\Sigma}{N}{\Lambda}$ be
bimodules which induce a singular equivalence of Morita type with
level between $\Lambda$ and~$\Sigma$.  Let $C$ and~$C'$ be maximal
Cohen--Macaulay modules over~$\Lambda$.  Then for every positive
integer~$n$, the map
\[
\Ext_\Lambda^n(C, C')
\xrightarrow{N \tl -}
\Ext_\Sigma^n(N \tl C, N \tl C')
\]
is an isomorphism.
\end{prop}
\begin{proof}
The idea is to translate the two Ext groups to stable Hom groups by
using Lemma~\ref{lem:stHom=ext}, and then use the equivalence of
stable categories of Cohen--Macaulay modules from
Proposition~\ref{prop:semtl->sCM}.

Let
\[
\pi_n \colon 0 \to K_n \to P_{n-1} \to \cdots \to P_0 \to C \to 0
\]
be the beginning of a projective resolution of~$C$ with $K_n$ as
$n$-th syzygy.  By Lemma~\ref{lem:semtl-functor-properties}, the
sequence $N \tl \pi_n$, which is obtained by applying the functor $N \tl
-$ to all objects and maps in $\pi_n$, is the beginning of a projective
resolution of the $\Sigma$-module $N \tl C$, with $N \tl K_n$ as the
$n$-th syzygy.

Since $C$ and $C'$ are maximal Cohen--Macaulay modules, we deduce that
$N \tl C$, $N \tl C'$, $K_n$ and $N \tl K_n$ are also maximal
Cohen--Macaulay modules, by using Lemma~\ref{lem:cm<->syzygy} and
Lemma~\ref{lem:tensor-functor-cm}.  We form the following commutative
diagram of $k$-vector spaces.
\[
\xymatrix@C=5em{
{\stHom_\Lambda(K_n, C')}
\ar[r]^-{N \tl -}_-{\iso}
\ar[d]_{\pi_n^*}^{\iso}
&
{\stHom_\Sigma(N \tl K_n, N \tl C')}
\ar[d]^{(N \tl \pi_n)^*}_{\iso}
\\
{\Ext^n_\Lambda(C, C')}
\ar[r]^-{N \tl -}
&
{\Ext^n_\Sigma(N \tl C, N \tl C')}
}
\]
The vertical maps are isomorphisms by Lemma~\ref{lem:stHom=ext}, and
the map in the top row is an isomorphism by
Proposition~\ref{prop:semtl->sCM}.  Therefore the map in the bottom row
is also an isomorphism, and this concludes the proof.
\end{proof}

Finally, we come to the main result of this section, where we show
that if two Gorenstein algebras $\Lambda$ and~$\Sigma$ are singularly
equivalent of Morita type with level, then for every extension group
(of sufficiently high degree) over~$\Lambda$, there is an isomorphic
extension group over~$\Sigma$.

\begin{prop}
\label{prop:ext-iso}
Let $\Lambda$ and~$\Sigma$ be finite-dimensional Gorenstein
$k$-algebras which are singularly equivalent of Morita type with
level, and let $\bimod{\Lambda}{M}{\Sigma}$
and~$\bimod{\Sigma}{N}{\Lambda}$ be bimodules which induce a singular
equivalence of Morita type with level between $\Lambda$ and~$\Sigma$.
Let
\[
d = \max \{ \id_\Lambda \Lambda, \id_\Sigma \Sigma \}
\]
be the maximum of the injective dimensions of $\Lambda$ and~$\Sigma$.
Then for every integer $n > d$, we have $k$-vector space isomorphisms
\begin{align*}
\Ext_\Lambda^n(A, B)
&\xrightarrow[N \tl -]{\iso}
\Ext_\Sigma^n(N \tl A, N \tl B)
&&
\text{for $\Lambda$-modules $A$ and $B$,} \\
\Ext_\Sigma^n(A', B')
&\xrightarrow[M \ts -]{\iso}
\Ext_\Lambda^n(M \ts A', M \ts B')
&&
\text{for $\Sigma$-modules $A'$ and $B'$.}
\end{align*}
\end{prop}
\begin{proof}
Let $A$ and~$B$ be $\Lambda$-modules, and let
\[
\pi \colon
\cdots \to
P_1 \to
P_0 \to
A \to
0
\qquad\text{and}\qquad
\tau \colon
\cdots \to
Q_1 \to
Q_0 \to
B \to
0
\]
be projective resolutions.  Then by
Lemma~\ref{lem:semtl-functor-properties}, the sequences $N \tl \pi$
and $N \tl \tau$ are projective resolutions of the $\Sigma$-modules $N
\tl A$ and $N \tl B$.  We form the following commutative diagram,
where $\rho_d$ is the $d$-th rotation map with respect to the
resolutions $\pi$ and $\tau$, and $\rho'_d$ the $d$-th rotation map
with respect to the resolutions $N \tl \pi$ and $N \tl \tau$.  These
maps are isomorphisms by Lemma~\ref{lem:rotation}.
\[
\xymatrix@C=5em{
{\Ext^n_\Lambda(A, B)}
\ar[r]^-{N \tl -}
\ar[d]_{\rho_d}^\iso
&
{\Ext^n_\Sigma(N \tl A, N \tl B)}
\ar[d]^{\rho'_d}_\iso
\\
{\Ext^{n+d}_\Lambda(\syzygy_\Lambda^d(A), \syzygy_\Lambda^d(B))}
\ar[r]^-{N \tl -}_-\iso
&
{\Ext^{n+d}_\Sigma(N \tl \syzygy_\Lambda^d(A), N \tl \syzygy_\Lambda^d(B))}
}
\]
By Lemma~\ref{lem:cm<->syzygy}, the syzygies $\syzygy_\Lambda^d(A)$
and $\syzygy_\Lambda^d(B)$ are maximal Cohen--Macaulay modules, and
then by Proposition~\ref{prop:ext-iso-cm}, the map $N \tl -$ in the
bottom row is an isomorphism.  It follows that the map $N \tl -$ in
the top row is an isomorphism.  This gives the first of the two
isomorphisms we want.  The second isomorphism follows by symmetry.
\end{proof}

\section{Hochschild cohomology rings}
\label{sec:hh}

In this section, we define the Hochschild cohomology ring
$\HH*(\Lambda)$ of an algebra~$\Lambda$, and we show that if two
Gorenstein $k$-algebras are singularly equivalent of Morita type with
level, then their Hochschild cohomology rings are isomorphic in almost
all degrees.

We first introduce some notation for rings of extensions.  If
$\Lambda$ is a $k$-algebra and $A$ a $\Lambda$-module, then we define
\[
\extring{\Lambda}{A}
= \Ext_\Lambda^*(A, A)
= \Dsum_{n \ge 0} \Ext_\Lambda^n(A, A).
\]
That is, $\extring{\Lambda}{A}$ denotes the graded $k$-algebra which
is the direct sum of all extension groups of $A$ by itself, with
multiplication given by Yoneda product.

We are interested in the ``asymptotic'' behaviour of such graded rings
of extensions; that is, we want to find isomorphisms which hold in all
degrees above some finite bound.  Given an extension ring
$\extring{\Lambda}{A}$, we therefore consider the graded ideals of the
form
\[
\extring[>d]{\Lambda}{A}
= \Dsum_{n > d} \Ext_\Lambda^n(A, A)
\]
for some integer~$d$.  We use the term \defterm{rng} for a ``ring
without identity''.  The object $\extring[>d]{\Lambda}{A}$ is thus a
graded rng.  In order to study the asymptotic behaviour of extension
rings, the appropriate morphisms to look at are the morphisms of
graded rngs between objects of the form $\extring[>d]{\Lambda}{A}$.

We define the Hochschild cohomology of an algebra as the extension
ring of the algebra over its enveloping algebra.

\begin{defn}
Let $\Lambda$ be a finite-dimensional $k$-algebra.  The
\defterm{Hochschild cohomology ring} of $\Lambda$ is the extension
ring $\HH*(\Lambda) = \extring{\e\Lambda}{\Lambda}$.
\end{defn}

Hochschild cohomology was first defined by G.~Hochschild
in~\cite{hochschild}.  The original definition uses the bar
resolution.  We follow the definition in~\cite{ce}, where Hochschild
cohomology is given by extension groups.  Since we have assumed that
$k$ is a field, this definition is equivalent to the original one.
More generally, the two definitions are equivalent whenever $\Lambda$
is projective over $k$ (see~\cite[IX, \S6]{ce}).

We now turn to the problem of showing that singular equivalences of
Morita type with level between Gorenstein algebras preserve Hochschild
cohomology in almost all degrees.  We need the following diagram
lemma, known as the ``$3 \times 3$ splice''.

\begin{lem} \cite[Lemma VIII.3.1]{maclane}
\label{lem:3x3-splice}
Let $R$ be a ring, and let
\[
\xymatrix@C=.1em@R=.1em{
             & \eta'\colon     && \eta\colon     && \eta''\colon \\
\eta_A\colon & A' \ar[rr]\ar[dd] && A \ar[rr]\ar[dd] && A'' \ar[dd]   \\\\
\eta_B\colon & B' \ar[rr]\ar[dd] && B \ar[rr]\ar[dd] && B'' \ar[dd]   \\\\
\eta_C\colon & C' \ar[rr]       && C \ar[rr]       && C''          \\\\
}
\]
be a commutative diagram of $R$-modules, where the three rows
$\eta_A$, $\eta_B$ and $\eta_C$, as well as the three columns $\eta'$,
$\eta$ and $\eta''$, are short exact sequences.  Then the elements in
the extension group $\Ext_R^1(C'', A')$ represented by the composition
$\eta_A \extcomp \eta''$ and by the composition $\eta' \extcomp
\eta_C$ are the additive inverses of each other:
\[
[\eta_A \extcomp \eta''] = - [\eta' \extcomp \eta_C].
\]
\end{lem}

If two bimodules $\bimod{\Lambda}{M}{\Sigma}$
and~$\bimod{\Sigma}{N}{\Lambda}$ induce a singular equivalence of
Morita type with level between algebras $\Lambda$ and~$\Sigma$, then
the $\e\Lambda$-module $M \ts N$ is a syzygy of~$\Lambda$.  In the
following lemma, we use Lemma~\ref{lem:3x3-splice} to show that under
certain assumptions, the tensor functors $(M \ts N) \tl -$ and $- \tl
(M \ts N)$ induce isomorphisms of Ext groups in almost all degrees.
This is afterwards used in the proof of Theorem~\ref{thm:hh-iso}.

\begin{lem}
\label{lem:tensor-syzygy}
Let $\Lambda$ be a finite-dimensional Gorenstein $k$-algebra, and let
$U$ be a $\e\Lambda$-module which is projective as a left
$\Lambda$-module and as a right $\Lambda$-module.  Let $d \ge 2 \cdot
\id_\Lambda \Lambda$.  Let $K$ be an $i$-th syzygy of $\Lambda$ as
$\e\Lambda$-module, for some $i < d$.  Then the maps
\[
K \tl - \colon
\extring[>d]{\e\Lambda}{U} \to
\extring[>d]{\e\Lambda}{K \tl U}
\quad \text{and} \quad
- \tl K \colon
\extring[>d]{\e\Lambda}{U} \to
\extring[>d]{\e\Lambda}{U \tl K}
\]
are graded rng isomorphisms.
\end{lem}
\begin{proof}
We show that the map $K \tl -$ is an isomorphism; the proof for $- \tl
K$ is similar.  Let
\[
\pi\colon \cdots \to P_1 \to P_0 \to \Lambda \to 0
\]
be a projective resolution of $\Lambda$ as $\e\Lambda$-module, with
$K$ as the $i$-th syzygy, and let
\[
\sigma \colon
\cdots \to P_1 \tl U \to P_0 \tl U \to U \to 0.
\]
be the result of applying the functor $- \tl U$ to the sequence~$\pi$
and identifying $\Lambda \tl U$ with~$U$ in the last term.  This
sequence is exact since $U$ is projective as left module, and every
$P_j \tl U$ is projective since $U$ is projective as right module.
Thus, $\sigma$ is a projective resolution of~$U$, and $K \tl U$ is an
$i$-th syzygy of~$U$.

By Lemma~\ref{lem:gorenstein-envalg}, the enveloping algebra
$\e\Lambda$ of $\Lambda$ is Gorenstein, and we have $\id_{\e\Lambda}
\e\Lambda \le 2 \cdot \id_\Lambda \Lambda \le d$.  Then by
Lemma~\ref{lem:rotation}, the $i$-th rotation map
\[
\rho_i \colon
\extring[>d]{\e\Lambda}{U} \to \extring[>d]{\e\Lambda}{K \tl U}
\]
(with respect to the resolution $\sigma$) is a graded rng isomorphism.
We show that the map $K \tl -$ is an isomorphism by showing that it is
equal to the map~$\rho_i$, up to sign.  More precisely, we show that
for any homogeneous element $[\eta] \in \extring[>d]{\e\Lambda}{U}$ of
degree $n > d$, we have
\[
K \tensor_\Lambda [\eta] = (-1)^{in} \cdot \rho_i([\eta]).
\]

Let $[\eta] \in \extring[>d]{\e\Lambda}{U}$ be a homogeneous element
of degree $n > d$ represented by an exact sequence
\[
\eta \colon
0 \to U \to E_n \to \cdots \to E_1 \to U \to 0.
\]
We can assume without loss of generality that all the modules $E_j$
are projective as left $\Lambda$-modules and as right
$\Lambda$-modules.  Let
\begin{align*}
\pi_i&\colon 0 \to K \to P_{i-1} \to \cdots \to P_0 \to \Lambda \to 0 \\
\sigma_i&\colon
0 \to K \tl U \to P_{i-1} \tl U \to \cdots \to P_0 \tl U \to U \to 0
\end{align*}
be truncations of the projective resolutions $\pi$ and $\sigma$.
\begin{figure}
\[
\xymatrix@C=1em@R=2em{
& 0 \ar[d] & 0 \ar[d] & & 0 \ar[d] & 0 \ar[d] \\
0 \ar[r] & K \tl U \ar[r]\ar[d] & K \tl E_n \ar[r]\ar[d] & \cdots \ar[r] & K \tl E_1 \ar[r]\ar[d] & K \tl U \ar[r]\ar[d] & 0 \\
0 \ar[r] & P_{i-1} \tl U \ar[r]\ar[d] & P_{i-1} \tl E_n \ar[r]\ar[d] & \cdots \ar[r] & P_{i-1} \tl E_1 \ar[r]\ar[d] & P_{i-1} \tl U \ar[r]\ar[d] & 0 \\
  & \vdots \ar[d] & \vdots \ar[d] & \ddots & \vdots \ar[d] & \vdots \ar[d] & \\
0 \ar[r] & P_0 \tl U \ar[r]\ar[d] & P_0 \tl E_n \ar[r]\ar[d] & \cdots \ar[r] & P_0 \tl E_1 \ar[r]\ar[d] & P_0 \tl U \ar[r]\ar[d] & 0 \\
0 \ar[r] & U \ar[r]\ar[d] & E_n \ar[r]\ar[d] & \cdots \ar[r] & E_1 \ar[r]\ar[d] & U \ar[r]\ar[d] & 0 \\
& 0  & 0  & & 0  & 0
}
\]
\caption{Commutative diagram used in the proof of
  Lemma~\ref{lem:tensor-syzygy}.}
\label{fig:tensor}
\end{figure}
We construct the commutative diagram in Figure~\ref{fig:tensor} by
tensoring $\pi_i$ with~$\eta$ over $\Lambda$ and identifying $\Lambda
\tl -$ with the identity in the last row.  The rows and columns of the
diagram are exact sequences.

The bottom row in the diagram is the sequence~$\eta$, the top row is
the sequence $K \tl \eta$, and the first and the last column are both
equal to the sequence $\sigma_i$.  By using Lemma~\ref{lem:3x3-splice}
repeatedly, we get the equality
\[
[(K \tensor_\Lambda \eta) \extcomp \sigma_i] = (-1)^{in} [\sigma_i \extcomp \eta]
\]
in the extension group $\Ext_{\e\Lambda}^{n+i}(U, K \tensor_\Lambda
U)$.  By the definition of the rotation map~$\rho_i$, we then get
\[
K \tensor_\Lambda [\eta]
= [K \tensor_\Lambda \eta]
= (-1)^{in} \cdot \rho_i([\eta]).
\]
Since the map $\rho_i$ is an isomorphism, this means that the map $K
\tensor_\Lambda -$ is an isomorphism as well.
\end{proof}

We now show that a singular equivalence of Morita type with level
between Gorenstein $k$-algebras preserves the Hochschild cohomology in
almost all degrees.  A weaker form of this result, stating that a
singular equivalence of Morita type preserves Hochschild cohomology
groups in almost all degrees (but not necessarily the ring structure
of the cohomology), appears in~\cite[Remark 4.3]{zz}.

\begin{thm}
\label{thm:hh-iso}
Let $\Lambda$ and~$\Sigma$ be finite-dimensional Gorenstein
$k$-algebras which are singularly equivalent of Morita type with
level.  Then we have the following.
\begin{enumerate}
\item The Hochschild cohomology rings $\HH{*}(\Lambda)$ and
$\HH{*}(\Sigma)$ are isomorphic in almost all degrees, with
isomorphisms that respect the ring structure.
\item Let $\bimod{\Lambda}{M}{\Sigma}$
and~$\bimod{\Sigma}{N}{\Lambda}$ be bimodules which induce a singular
equivalence of Morita type with level~$l \ge 1$ \textup{(}see
Remark~\ref{rem:hh-iso-level}\textup{)} between $\Lambda$
and~$\Sigma$, and let $d = \max \{ l, 2 \cdot \id_\Lambda \Lambda, 2
\cdot \id_\Sigma \Sigma \}$.  Then there are isomorphisms
\[
\xymatrix{
\HH{>d}(\Lambda) \ar[d]_{N \tl - \tl M}^{\iso} \ar[r]^-{\rho_l}_-{\iso} &
\extring[>d]{\e\Lambda}{M \ts \Sigma \ts N} \\
\extring[>d]{\e\Sigma}{N \tl \Lambda \tl M} &
\HH{>d}(\Sigma) \ar[u]_{M \ts - \ts N}^{\iso} \ar[l]^-{\rho'_l}_-{\iso}
}
\]
of graded rngs, where the maps $\rho_l$ and~$\rho'_l$ are rotation
maps.
\end{enumerate}
\end{thm}
\begin{proof}
We show part~(2).  Part~(1) then follows directly.

Since $M$ and~$N$ induce a singular equivalence of Morita type with
level~$l$, the module $M \ts \Sigma \ts N \iso M \ts N$ is an $l$-th
syzygy of~$\Lambda$ as a $\e\Lambda$-module.  Let
\[
\rho_l \colon \HH{>d}(\Lambda) \to \extring[>d]{\e\Lambda}{M \ts \Sigma \ts N}
\]
be the $l$-th rotation map with respect to a projective resolution of
$\Lambda$ with $M \ts \Sigma \ts N$ as the $l$-th syzygy.  By
Lemma~\ref{lem:gorenstein-envalg}, the enveloping algebras
$\e\Lambda$ and~$\e\Sigma$ are Gorenstein algebras, and we have
$\id_{\e\Lambda} \e\Lambda \le 2 \cdot \id_\Lambda \Lambda$ and
$\id_{\e\Sigma} \e\Sigma \le 2 \cdot \id_\Sigma \Sigma$.  By
Lemma~\ref{lem:rotation}, the rotation map $\rho_l$ is an isomorphism,
since
\[
\max \{ l, \id_{\e\Lambda} \e\Lambda, \id_{\e\Sigma} \e\Sigma \}
\le \max \{ l, 2 \cdot \id_\Lambda \Lambda, 2 \cdot \id_\Sigma \Sigma \}
= d.
\]
We can similarly define the rotation map~$\rho'_l$ and show that it is
an isomorphism.

We now show that the maps $N \tl - \tl M$ and $M \ts - \ts N$ are
isomorphisms.  For any $n > d$, we can make the following diagram:
\begin{equation}
\label{eqn:hh-diagram}
\vcenter{
\xymatrix{
\HH{n}(\Lambda) \ar[d]_{N \tl - \tl M} \ar[r]^-{\rho_l}_-{\iso} &
\extring[n]{\e\Lambda}{M \ts \Sigma \ts N} \\
\extring[n]{\e\Sigma}{N \tl \Lambda \tl M} &
\HH{n}(\Sigma) \ar[u]_{M \ts - \ts N} \ar[l]^-{\rho'_l}_-{\iso}
}}
\end{equation}
Consider the map $N \tl - \tl M$ in this diagram.  We construct the
following commutative diagram with this map at the top:
\[
\xymatrix@C=6em{
\HH{n}(\Lambda)
\ar[r]^-{N \tl - \tl M}
\ar[d]_{(M \ts N) \tl -}^{\iso}
&
\extring[n]{\e\Sigma}{N \tl \Lambda \tl M}
\ar[d]^{M \ts - \ts N}
\\
\extring[n]{\e\Lambda}{M \ts N \tl \Lambda}
\ar[r]_-{- \tl (M \ts N)}^-{\iso}
&
\extring[n]{\e\Lambda}{M \ts N \tl \Lambda \tl M \ts N}
}
\]
By Lemma~\ref{lem:tensor-syzygy}, the maps $(M \ts N) \tl -$ and $-
\tl (M \ts N)$ in this diagram are isomorphisms, since $M \ts N$ is an
$l$-th syzygy of $\Lambda$ as $\e\Lambda$-module.  Therefore, the map
$N \tl - \tl M$ in diagram~\eqref{eqn:hh-diagram} is a monomorphism.
By a similar argument, the map $M \ts - \ts N$ in
diagram~\eqref{eqn:hh-diagram} is a monomorphism.  Since
$\HH{n}(\Lambda)$ and $\HH{n}(\Sigma)$ are finite-dimensional
over~$k$, it follows that these monomorphisms must be isomorphisms.
\end{proof}

\begin{rem}
\label{rem:hh-iso-level}
In Theorem~\ref{thm:hh-iso}~(2), we assumed that the level~$l$ is
positive.  The reason for this is that if we had allowed $l=0$, then
we could not have made the rotation maps $\rho_l$ and $\rho'_l$.  This
assumption does not strongly affect the applicability of the theorem,
since any equivalence with level~$0$ implies the existence of an
equivalence with level~$1$.  In general, if two bimodules
$\bimod{\Lambda}{M}{\Sigma}$ and $\bimod{\Sigma}{N}{\Lambda}$ induce a
singular equivalence of Morita type with level~$l$ between algebras
$\Lambda$ and $\Sigma$, then the bimodules $\syzygy_{\Lambda \tk
  \opposite{\Sigma}}^1(M)$ and $N$ induce a singular equivalence of
level $l+1$ between $\Lambda$ and $\Sigma$.
\end{rem}

\section{Finite generation}
\label{sec:fg}

Support varieties for modules over artin algebras were defined by
Snashall and Solberg in~\cite{ss}, using the Hochschild cohomology
ring.  In~\cite{ehsst}, Erdmann, Holloway, Snashall, Solberg and
Taillefer defined two finite generation conditions \textbf{Fg1}
and~\textbf{Fg2} for the Hochschild cohomology ring of an algebra.
These conditions ensure that the support varieties for modules over
the given algebra have good properties.  In~\cite{es}, these
conditions were reformulated as a new condition called~\textbf{(Fg)}
which is equivalent to the combination of \textbf{Fg1} and
\textbf{Fg2}.  We use the definition from~\cite{es}.

In this section, we describe the finite generation
condition~\fgtext{(Fg)}.  We then show the main result of this paper
(Theorem~\ref{thm:main}): A singular equivalence of Morita type with
level between finite-dimensional Gorenstein $k$-algebras preserves the
\fgtext{(Fg)} condition.

In order to define the \fgtext{(Fg)} condition, we first describe a way
to view extension rings over an algebra as modules over the Hochschild
cohomology ring.  Let $\Lambda$ be a finite-dimensional $k$-algebra
and $A$ a $\Lambda$-module.  We define a graded ring homomorphism
\[
\varphi_A \colon
\HH*(\Lambda) \to \extring{\Lambda}{A}
\]
as follows.  A homogeneous element of $\HH*(\Lambda)$ can be
represented by an exact sequence
\[
\eta \colon
0 \to
\Lambda \to
E \to
P_n \to
\cdots \to
P_0 \to
\Lambda \to
0
\]
of $\e\Lambda$-modules, where each $P_i$ is projective.  Viewed as a
sequence of right $\Lambda$-modules, this sequence splits.  The
complex
\[
\eta \tl A \colon
0 \to
\Lambda \tl A \to
E \tl A \to
P_n \tl A \to
\cdots \to
P_0 \tl A \to
\Lambda \tl A \to
0
\]
is therefore an exact sequence.  By composition with the isomorphism
$\mu_A\colon \Lambda \tl A \to A$ and its inverse, we get an extension
\[
\mu_A \extcomp (\eta \tl A) \extcomp \mu_A^{-1}
\colon
0 \to
A \to
E \tl A \to
P_n \tl A \to
\cdots \to
P_0 \tl A \to
A \to
0
\]
of $A$ by itself, and thus a representative of a homogeneous element
in the extension ring $\extring{\Lambda}{A}$.  The map $\varphi_A$ is
defined by the action
\[
\varphi_A([\eta]) = [\mu_A \extcomp (\eta \tl A) \extcomp \mu_A^{-1}]
\]
on homogeneous elements.  By the map~$\varphi_A$, the graded ring
$\extring{\Lambda}{A}$ becomes a graded $\HH*(\Lambda)$-module.

\begin{defn}
Let $\Lambda$ be a finite-dimensional $k$-algebra.  We say that
$\Lambda$ satisfies the \fgtext{(Fg)} condition if the following holds.
\begin{enumerate}
\item The ring $\HH*(\Lambda)$ is Noetherian.
\item The $\HH*(\Lambda)$-module $\extring{\Lambda}{\Lambda/\rad
  \Lambda}$ is finitely generated.  (The module structure is given by
the map $\varphi_{\Lambda/\rad \Lambda}$, as described above.)
\end{enumerate}
\end{defn}

By~\cite[Proposition~5.7]{varieties-survey}, the \fgtext{(Fg)} condition
as defined here is equivalent to the combination of the conditions
\fgtext{Fg1} and~\fgtext{Fg2} defined in~\cite{ehsst}.

The following result describes why Gorenstein algebras are important
in connection with the \fgtext{(Fg)} condition.

\begin{thm}\cite[Theorem~1.5~(a)]{ehsst}
\label{thm:fg=>gorenstein}
If an algebra satisfies the \fgtext{(Fg)}~condition, then it is a
Gorenstein algebra.
\end{thm}

Our aim is to show that if two Gorenstein $k$-algebras are singularly
equivalent of Morita type with level, then the \fgtext{(Fg)} condition
holds for one of the algebras if and only if it holds for the other.
We use the following result, which describes a relation between two
algebras ensuring that \fgtext{(Fg)} for one of the algebras implies
\fgtext{(Fg)} for the other.

\begin{prop}
\label{prop:fg}
Let $\Lambda$ and~$\Sigma$ be finite-dimensional $k$-algebras.  Let $A
= \Lambda/\rad \Lambda$, and assume that we have a commutative diagram
\begin{equation}
\begin{gathered}
\label{eqn:fg-diagram}
\xymatrix@C=5em{
\HH{>d}(\Lambda)
\ar[r]^{\varphi_A}
\ar[d]^{f}_{\iso}
&
\extring[>d]{\Lambda}{A}
\ar[d]_{g}^{\iso}
\\
\HH{>d}(\Sigma)
\ar[r]^{\varphi_B}
&
\extring[>d]{\Sigma}{B}
}
\end{gathered}
\end{equation}
of graded rngs, for some $\Sigma$-module $B$ and some positive
integer~$d$, where the vertical maps $f$ and~$g$ are isomorphisms.
Assume that $\Sigma$ satisfies the \fgtext{(Fg)} condition.  Then
$\Lambda$ also satisfies~\fgtext{(Fg)}.
\end{prop}
\begin{proof}
This follows from Proposition~6.3 in~\cite{pss}.
\end{proof}

We are now ready to prove the main result of this paper.

\begin{thm}
\label{thm:main}
Let $\Lambda$ and~$\Sigma$ be finite-dimensional Gorenstein algebras
over the field~$k$.  Assume that $\Lambda$ and~$\Sigma$ are singularly
equivalent of Morita type with level.  Then $\Lambda$ satisfies
\fgtext{(Fg)} if and only if $\Sigma$ satisfies \fgtext{(Fg)}.
\end{thm}
\begin{proof}
We show that if $\Sigma$ satisfies \fgtext{(Fg)}, then $\Lambda$
satisfies \fgtext{(Fg)}.  The opposite implication then follows by
symmetry.  Let $\bimod{\Lambda}{M}{\Sigma}$
and~$\bimod{\Sigma}{N}{\Lambda}$ be bimodules which induce a singular
equivalence of Morita type with level~$l \ge 1$ (see
Remark~\ref{rem:hh-iso-level}) between $\Lambda$ and~$\Sigma$.  Let $d
= \max \{ l, 2 \cdot \id_\Lambda \Lambda, 2 \cdot \id_\Sigma \Sigma
\}$.  Let $A$ be the $\Lambda$-module $\Lambda/\rad \Lambda$.

The $\e\Lambda$-module $M \ts \Sigma \ts N \iso M \ts N$ is an $l$-th
syzygy of~$\Lambda$ as $\e\Lambda$-module.  Let $\pi$ be a projective
resolution of~$\Lambda$ with $M \ts \Sigma \ts N$ as the $l$-th
syzygy.  Then the complex $\pi \tl A$ is a projective resolution of
$\Lambda \tl A$, with $M \ts \Sigma \ts N \tl A$ as the $l$-th syzygy.
We construct the commutative diagram in
Figure~\ref{fig:hamburger-diagram}, where the maps $\rho_l$ and
$\rho'_l$ are the $l$-th rotation maps with respect to the resolutions
$\pi$ and $\pi \tl A$, respectively.  These maps are isomorphisms by
Lemma~\ref{lem:rotation}.  The map $M \ts - \ts N$ in the diagram is
an isomorphism by Theorem~\ref{thm:hh-iso}, and the map $M \ts -$ is
an isomorphism by Proposition~\ref{prop:ext-iso}.  The isomorphisms
$f$ and~$g$ are defined to be the appropriate compositions of the
other isomorphisms in the diagram.  By Proposition~\ref{prop:fg}, this
diagram shows that if the algebra $\Sigma$ satisfies \fgtext{(Fg)}, then
$\Lambda$ also satisfies \fgtext{(Fg)}.
\begin{figure}
\[
\xymatrix{
&
\HH{>d}(\Lambda)
\ar[r]^{- \tl A}
\ar[d]_{\rho_l}^{\iso}
\ar@(u,u)[rr]^{\varphi_A}
\ar@{-->} `l[dl] `[dd]_{f}^{\iso} [dd]
&
\extring[>d]{\Lambda}{\Lambda \tl A}
\ar[r]^{\iso}
\ar[d]_{\rho'_l}^{\iso}
&
\extring[>d]{\Lambda}{A}
\ar@{-->}[dd]^{g}_{\iso}
\\
&
\extring[>d]{\e\Lambda}{M \ts \Sigma \ts N}
\ar[r]^-{- \tl A}
&
\extring[>d]{\Lambda}{M \ts \Sigma \ts N \tl A}
\\
&
\HH{>d}(\Sigma)
\ar[u]^{M \ts - \ts N}_{\iso}
\ar[r]_-{- \ts (N \tl A)}
\ar@(d,d)[rr]_{\varphi_{N \ts A}}
&
\extring[>d]{\Sigma}{\Sigma \ts N \tl A}
\ar[r]_{\iso}
\ar[u]^{M \ts -}_{\iso}
&
\extring[>d]{\Sigma}{N \tl A}
}
\]
\caption{Commutative diagram used in the proof of
  Theorem~\ref{thm:main}.}
\label{fig:hamburger-diagram}
\end{figure}
\end{proof}

We now show that the assumption of both algebras being Gorenstein is
necessary in the above theorem.  Example~5.5 in~\cite{pss} contains
two singularly equivalent algebras where one algebra satisfies
\fgtext{(Fg)} and the other is not Gorenstein.  We use the same
algebras, and show that there exists a singular equivalence of Morita
type with level between them.

\begin{ex}
\label{ex:fg-not-preserved}
Let $\Lambda = kQ/\langle\rho\rangle$ and $\Sigma =
kR/\langle\sigma\rangle$ be $k$-algebras given by the following
quivers and relations:
\[
\setlength\arraycolsep{1em}
\begin{array}{ll}
Q \colon
\xymatrix{
1 \ar@(ul,dl)_{\alpha} \ar[r]^{\beta} &
2
}
&
\rho = \{ \alpha^2, \beta\alpha \} \\[1.5em]
R \colon
\xymatrix{
3 \ar@(ul,dl)_{\gamma}
}
&
\sigma = \{ \gamma^2 \}
\end{array}
\]
The tensor algebra $\Lambda \tk \opposite{\Sigma}$ has the following
quiver and relations:
\[
Q \times \opposite{R} \colon
\vcenter{
\xymatrix{
1 \times \opposite{3}
\ar@(ul,ur)^{\alpha \times \opposite{3}}
\ar@(ur,dr)^{1 \times \opposite\gamma}
\ar[d]^{\beta \times \opposite{3}} \\
2 \times \opposite{3}
\ar@(ur,dr)^{2 \times \opposite\gamma}
}}
\qquad
\mbox{\scriptsize
$\begin{Bmatrix}
(\alpha \times \opposite{3})^2,
(\beta \times \opposite{3})(\alpha \times \opposite{3}), \\
(1 \times \opposite{\gamma})^2,
(2 \times \opposite{\gamma})^2, \\
(\alpha \times \opposite{3})(1 \times \opposite{\gamma})
- (1 \times \opposite{\gamma})(\alpha \times \opposite{3}), \\
(\beta \times \opposite{3})(1 \times \opposite{\gamma})
- (2 \times \opposite{\gamma})(\beta \times \opposite{3})
\end{Bmatrix}$
}
\]
The tensor algebra $\Sigma \tk \opposite{\Lambda}$ has the following
quiver and relations:
\[
R \times \opposite{Q} \colon
\xymatrix{
3 \times \opposite{1}
\ar@(ul,dl)_{3 \times \opposite{\alpha}}
\ar@(dl,dr)_{\gamma \times \opposite{1}} &
3 \times \opposite{2}
\ar[l]_{3 \times \opposite{\beta}}
\ar@(dl,dr)_{\gamma \times \opposite{2}}
}
\qquad
\mbox{\scriptsize
$\begin{Bmatrix}
(3 \times \opposite{\alpha})^2,
(3 \times \opposite\alpha)(3 \times \opposite\beta), \\
(\gamma \times \opposite{1})^2,
(\gamma \times \opposite{2})^2, \\
(3 \times \opposite{\alpha})(\gamma \times \opposite{1})
- (\gamma \times \opposite{1})(3 \times \opposite{\alpha}), \\
(3 \times \opposite{\beta})(\gamma \times \opposite{2})
- (\gamma \times \opposite{1})(3 \times \opposite{\beta})
\end{Bmatrix}$
}
\]
Let $\bimod{\Lambda}{M}{\Sigma}$ and~$\bimod{\Sigma}{N}{\Lambda}$ be
bimodules given by the following representations over $Q \times
\opposite{R}$ and $R \times \opposite{Q}$, respectively:
\[
M \colon
\vcenter{
\xymatrix{
k^2
\ar@(ul,ur)^{\left(\begin{smallmatrix} 0 & 0 \\ 1 & 0 \end{smallmatrix}\right)}
\ar@(ur,dr)^{\left(\begin{smallmatrix} 0 & 0 \\ 1 & 0 \end{smallmatrix}\right)}
\ar[d]_{\left(\begin{smallmatrix} 0 & 0 \\ 1 & 0 \end{smallmatrix}\right)} \\
k^2
\ar@(ur,dr)^{\left(\begin{smallmatrix} 0 & 0 \\ 1 & 0 \end{smallmatrix}\right)}
}}
\qquad
N \colon
\xymatrix{
k^2
\ar@(ul,dl)_{\left(\begin{smallmatrix} 0 & 0 \\ 1 & 0 \end{smallmatrix}\right)}
\ar@(dl,dr)_{\left(\begin{smallmatrix} 0 & 0 \\ 1 & 0 \end{smallmatrix}\right)} &
0
\ar[l]
\ar@(dl,dr)
}
\]
We show that the bimodules $M$ and~$N$ induce a singular equivalence
of Morita type with level~$1$ between the algebras $\Lambda$
and~$\Sigma$.  We first check that these bimodules satisfy the first
two conditions in the definition.  Considering only the left or right
structure of $M$ and~$N$, we have the following four isomorphisms:
\begin{align*}
\leftmod{\Lambda}{M} &\iso \Lambda &
\rightmod{N}{\Lambda} &\iso e_2 \Lambda \\
\leftmod{\Sigma}{N} &\iso \Sigma^2 &
\rightmod{M}{\Sigma} &\iso \Sigma
\end{align*}
Thus the bimodules $M$ and~$N$ are projective when viewed as one-sided
(left or right) modules.

To check the last two conditions in the definition of singular
equivalence of Morita type with level, we compute the tensor products
$M \ts N$ and $N \tl M$ as representations of quivers, and check that
they are syzygies of $\Lambda$ and~$\Sigma$, respectively.  The
enveloping algebra $\e\Lambda$ has the following quiver and relations:
\[
Q \times \opposite{Q} \colon
\vcenter{
\xymatrix{
1 \times \opposite{1}
\ar@(ul,ur)^{\alpha \times \opposite{1}}
\ar@(ul,dl)_{1 \times \opposite\alpha}
\ar[d]^{\beta \times \opposite{1}}
&
1 \times \opposite{2}
\ar@(ul,ur)^{\alpha \times \opposite{2}}
\ar[l]_{1 \times \opposite\beta}
\ar[d]^{\beta \times \opposite{2}}
\\
2 \times \opposite{1}
\ar@(ul,dl)_{2 \times \opposite\alpha}
&
2 \times \opposite{2}
\ar[l]_{2 \times \opposite\beta}
}}
\qquad
\mbox{\scriptsize
$\begin{Bmatrix}
(\alpha \times \opposite{1})^2,
(\beta \times \opposite{1})(\alpha \times \opposite{1}), \\
(\alpha \times \opposite{2})^2,
(\beta \times \opposite{2})(\alpha \times \opposite{2}), \\
(1 \times \opposite\alpha)^2,
(1 \times \opposite\alpha)(1 \times \opposite\beta), \\
(2 \times \opposite\alpha)^2,
(2 \times \opposite\alpha)(2 \times \opposite\beta), \\
(\alpha \times \opposite{1})(1 \times \opposite\alpha)
- (1 \times \opposite\alpha)(\alpha \times \opposite{1}), \\
(\alpha \times \opposite{1})(1 \times \opposite\beta)
- (1 \times \opposite\beta)(\alpha \times \opposite{2}), \\
(\beta \times \opposite{1})(1 \times \opposite\alpha)
- (2 \times \opposite\alpha)(\beta \times \opposite{1}), \\
(\beta \times \opposite{1})(1 \times \opposite\beta)
- (2 \times \opposite\beta)(\beta \times \opposite{2})
\end{Bmatrix}$
}
\]
The tensor product $M \ts N$ is the $\e\Lambda$-module given by the
following representation over $Q \times \opposite{Q}$:
\[
M \ts N \colon
\vcenter{
\xymatrix{
k^2
\ar@(ul,ur)^{\left(\begin{smallmatrix} 0 & 0 \\ 1 & 0 \end{smallmatrix}\right)}
\ar@(ul,dl)_{\left(\begin{smallmatrix} 0 & 0 \\ 1 & 0 \end{smallmatrix}\right)}
\ar[d]_{\left(\begin{smallmatrix} 0 & 0 \\ 1 & 0 \end{smallmatrix}\right)}
&
0
\ar[l]
\ar[d]
\ar@(ul,ur)
\\
k^2
\ar@(ul,dl)_{\left(\begin{smallmatrix} 0 & 0 \\ 1 & 0 \end{smallmatrix}\right)}
&
0
\ar[l]
}}
\]
The algebra $\Lambda$ considered as a $\e\Lambda$-module has the
following representation over $Q \times \opposite{Q}$:
\[
\Lambda \colon
\vcenter{
\xymatrix{
k^2
\ar@(ul,ur)^{\left(\begin{smallmatrix} 0 & 0 \\ 1 & 0 \end{smallmatrix}\right)}
\ar@(ul,dl)_{\left(\begin{smallmatrix} 0 & 0 \\ 1 & 0 \end{smallmatrix}\right)}
\ar[d]_{\left(\begin{smallmatrix} 1 & 0 \end{smallmatrix}\right)}
&
0
\ar[l]
\ar[d]
\ar@(ul,ur)
\\
k
\ar@(ul,dl)_{0}
&
k
\ar[l]^{1}
}}
\]
There is an exact sequence
\[
0
\to M \ts N
\to \e\Lambda e_{1 \times \opposite{1}} \dsum \e\Lambda e_{2 \times \opposite{2}}
\to \Lambda
\to 0
\]
of $\e\Lambda$-modules, and thus $M \ts N$ is a first syzygy of
$\Lambda$.

The enveloping algebra $\e\Sigma$ has the following quiver and
relations:
\[
R \times \opposite{R} \colon
\vcenter{
\xymatrix{
3 \times \opposite{3}
\ar@(dl,dr)_{\gamma \times \opposite{3}}
\ar@(ul,dl)_{3 \times \opposite\gamma}
}}
\qquad
\mbox{\scriptsize
$\begin{Bmatrix}
(\gamma \times \opposite{3})^2,
(3 \times \opposite\gamma)^2, \\
(\gamma \times \opposite{3})(3 \times \opposite\gamma)
- (3 \times \opposite\gamma)(\gamma \times \opposite{3})
\end{Bmatrix}$
}
\]
The algebra $\Sigma$ considered as a $\e\Sigma$-module has the
following representation over $R \times \opposite{R}$:
\[
\Sigma \colon
\xymatrix{
k^2
\ar@(ul,dl)_{\left(\begin{smallmatrix} 0 & 0 \\ 1 & 0 \end{smallmatrix}\right)}
\ar@(dl,dr)_{\left(\begin{smallmatrix} 0 & 0 \\ 1 & 0 \end{smallmatrix}\right)}
}
\]
Its minimal projective resolution is
\[
\cdots \to \e\Sigma \to \e\Sigma \to \Sigma \to 0,
\]
with $\Sigma$ itself as every syzygy.  The tensor product $N \tl M$ is
isomorphic to $\Sigma$ as $\e\Sigma$-module; in particular, it is a
first syzygy of $\Sigma$.

We have now shown that the bimodules $M$ and~$N$ induce a singular
equivalence of Morita type with level~$1$ between the algebras
$\Lambda$ and~$\Sigma$.  The algebra $\Sigma$ satisfies the
\fgtext{(Fg)} condition, but $\Lambda$ does not, and is not even a
Gorenstein algebra.  This shows that the assumption of both algebras
being Gorenstein can not be removed in Theorem~\ref{thm:main}.
\end{ex}

For stable equivalences of Morita type (which are singular
equivalences of Morita type with level~$0$), we can, under some
conditions, remove the assumption of Gorensteinness.

\begin{cor}
\label{cor:stable-equivalence}
Let $\bimod{\Lambda}{M}{\Sigma}$ and~$\bimod{\Sigma}{N}{\Lambda}$ be
indecomposable bimodules that induce a stable equivalence of Morita
type between two finite-dimensional $k$-algebras $\Lambda$
and~$\Sigma$.  Assume that $\Lambda$ and $\Sigma$ have no semisimple
blocks and that $\Lambda/\rad \Lambda$ and $\Sigma/\rad \Sigma$ are
separable.  Then $\Lambda$ satisfies \fgtext{(Fg)} if and only if
$\Sigma$ satisfies \fgtext{(Fg)}.
\end{cor}
\begin{proof}
By~\cite[Corollary~3.1~(2)]{dugas-martinez-villa}, the assumptions in
the statement of the result imply that $(M \ts -, N \tl -)$ and $(N
\tl -, M \ts -)$ are adjoint pairs.  Then,
by~\cite[Corollary~4.6]{liu-xi}, it follows that $\Lambda$ is a
Gorenstein algebra if and only if $\Sigma$ is a Gorenstein algebra.
The result now follows from Theorem~\ref{thm:main}.
\end{proof}

\bibliographystyle{alpha}
\bibliography{bibliography}

\end{document}